\documentclass[11pt]{amsart}
\usepackage{geometry}          
\usepackage{graphicx}
\usepackage{amssymb, amsmath, latexsym, amscd}
\usepackage{epstopdf}
\usepackage{pinlabel}
\newtheorem{theorem}{Theorem}

\newtheorem{corollary}[theorem]{Corollary}
\newtheorem{lemma}[theorem]{Lemma}
\newtheorem{proposition}[theorem]{Proposition}
\newtheorem*{definition}{Definition}

\newtheorem*{CVtheorem}{Theorem}
\newtheorem*{claim}{Claim}

\newcommand{\G}{\Gamma}

\newcommand{\Z}{\mathbb Z}

\newcommand{\AG}{A_\G}

\newcommand{\KW}{K_W}       
\newcommand{\KS}{K_{n,k}}    
\newcommand{\KO}{P_{n,k}}    
\newcommand{\LO}{D_{n,k}}     
\newcommand{\out}{{\rm Out}}
\newcommand{\aut}{{\rm Aut}}

\title[Automorphisms of RAAGs]{Automorphisms of two-dimensional RAAGS and partially symmetric automorphisms of free groups}

\author[Bux, Charney, and Vogtmann]{Kai-Uwe Bux, Ruth Charney,  and Karen Vogtmann}

\thanks {R. Charney was partially supported by NSF grant DMS 0705396.  K. Vogtmann was partially supported by NSF grant DMS 0705960}

\begin{document}
\maketitle

\begin{abstract}  We compute the VCD of the group of partially symmetric outer automorphisms of a free group.  We use this to obtain new upper and lower bounds on the VCD of the outer automorphism group of a 2-dimensional right-angled Artin group.  In the case of a right-angled Artin group with defining graph a tree, the bounds agree.
\end{abstract}

\section{Introduction}

This paper continues the study of outer automorphism groups of right-angled Artin groups begun in \cite{CCV} and \cite{ChVo}. It was shown  in \cite{ChVo}  that, for any finite simplicial graph $\G$ and associated right-angled Artin group $A_\G$,  the group $\out(A_\G)$  has finite virtual cohomological dimension.  In the two-dimensional case, i.e. when $\G$ is connected and has no triangles,  upper and lower bounds on this dimension were given in \cite{CCV}.   In the present paper we reconsider the two-dimensional case  and improve both the upper and the lower bounds.  If $\G$ is a tree these  bounds agree, giving:

\medskip\noindent{\bf Theorem~\ref{Tree}.} {\em   If $\G$ is a tree, the VCD  of $\out(\AG)$ is equal to $e+2\ell-3$, where $e$ is the number of edges of the tree and $\ell$ is the number of leaves.}
\medskip

 One reason the tree case is of particular interest is that right-angled Artin groups based on trees are fundamental groups of irreducible 3-manifolds; in fact, by a theorem of Droms a right-angled Artin group is the fundamental group of an irreducible 3-manifold if and only if it is based on a graph which is a tree or a triangle \cite{Dro}. 
 
Recall that a graph is {\it 2-connected} if it has no separating nodes. If $\G$ has no cycles of length less than 5 but is not a tree, the upper and lower bounds differ by the Euler characteristic of $\G$:

\medskip\noindent{\bf Theorem~\ref{NoShortCycles}.} {\em   If $\G$ is connected with no triangles or squares but is not a tree, the VCD of $\out(\AG)$ satisfies  $$ \pi  + 2\ell-1 \leq\hbox{\sc vcd}(\out(A_\G)) \leq \pi+2\ell-1-2\chi(\G),$$  where $\ell$ is the number of leaves of $\G$, $\pi$ is the number of maximal 2-connected subgraphs, and $\chi(\G)$ is the Euler characteristic of $\G$. }
\medskip

In particular, if $\G$ is connected with Euler characteristic 0, this gives the exact VCD of $\out(A_\G)$.  Bounds on the VCD  for graphs $\G$ with no triangles are given in full generality in Theorems~\ref{LBgeneral} and \ref{UBgeneral}.

    We find our lower bound by constructing a free abelian subgroup of $\out(\AG)$. For the upper bound, we reduce the problem to that of finding the VCD of the subgroup $P\Sigma(n,k)$ of   $\out(F_n)$ generated by automorphisms which send the first $k$ generators to conjugates of themselves.  If $n=k$, this is known as the {\it pure symmetric group} and was shown by Collins \cite{Col} to have virtual cohomological dimension equal to $n-2$.  We determine the exact VCD for any $k\geq 1$ as follows:
    
\medskip\noindent{\bf Theorem~\ref{PSigma}.} {\em For any $k\geq 1$, the group  $P\Sigma(n,k)$ has VCD equal to $2n-k-2$.}
\medskip
    
 $P\Sigma(n,k)$ has an obvious free abelian subgroup of rank $2n-k-2$, giving a lower bound on its VCD.  We show that this is equal to the VCD by finding a contractible complex of dimension $2n-k-2$ on which $P\Sigma(n,k)$ acts properly.  
 This subcomplex is a natural deformation retract of the {\it minimal subcomplex} $K_W$ of the spine of outer space associated to the set of cyclic words $W=\{x_1,\ldots,x_k\}$. These minimal subcomplexes were defined and shown to be contractible for any set $W$ of cyclic words in \cite{CuVo}.   
 
 The authors would like to thank Adam Piggott for helpful conversations.


\section{Two-dimensional right-angled Artin groups} 

A right-angled Artin group $\AG$ is two-dimensional if the maximal rank of an abelian subgroup is two, or equivalently, if the defining graph $\G$ has no triangles.  In this paper, we assume, in addition, that $\G$ is a connected graph and is not the star of a single vertex.  (If $\G$ is a star, then $\AG$ is the direct product of $\Z$ and a free group and the VCD is easily computed directly.)

 In this section, we establish upper and lower bounds for the VCD of  $\out(\AG)$.

\subsection{Lower bound for  triangle-free $\G$}
We find a  lower bound on the VCD of $\out(\AG)$  by constructing a free abelian subgroup.  This subgroup generally has larger rank than the one constructed  in \cite{CCV}.

We begin by recalling from \cite{CCV} the construction of a subgraph $\G_0$ of $\G$.  Nodes of $\G$ are  partially ordered by the relation $v\leq w$ if $lk(v)\subseteq lk(w)$ and two nodes are equivalent if they have the same link.  Choose a node in each maximal equivalence class,  and let $\G_0$ be the (connected) subgraph of $\G$ spanned by these nodes. For example, if $\G$ has no squares as well as no triangles, then every interior node of $\G$ is maximal, and $\G_0$ is the graph obtained from $\G$ by pruning off all of its leaves.  More generally, the assumption that $\G$ is not a star guarantees that leaves of $\G$ never lie in $\G_0$.

\newcommand{\wplus}{\hat{w}}
\newcommand{\vp}{\hat{v}} 
\newcommand{\ov}{\hat{v}} 

We will now describe a method of finding a large collection of commuting automorphisms; these will be partial conjugations by nodes in $\G_0$ and transvections of nodes in $\G_0$ onto nodes in $\G-\G_0$.

Fix an edge $e_0$ of $\G_0,$ with endpoints $v_0$ and $w_0$, and let $T_0$ be a maximal tree in $\G_0$ containing $e_0$.  Orient each edge of $T_0-e_0$ towards $e_0$.  For each node $v$ of $\G_0$,  let $\vp$ denote the (unique) node of $T_0$ which is adjacent to $v$ and is not on an incoming edge. For the endpoints of $e_0$, we have $\vp_0=w_0$ and $\wplus_0=v_0$.

For nodes  $v$ of $\G-\G_0$,  choose a node $\ov\in \G_0$ with $v\leq \ov$.
 
 Let $G(e_0,T_0)$ be the   subgroup of $\out(\AG)$ generated by the following automorphisms:
\begin{enumerate}
\item For each $v\in \G_0$, partial conjugations by $\vp$ on components of $\G-\{v\}$ which do not contain $\vp$.  
\item For each leaf $u$, right transvections $u\mapsto uv$ and  $u\mapsto u\vp$, where $v$ is the (unique) node adjacent to $u$.
\item For each $v\in \G-\G_0$ which is not a leaf, right and left transvections $v\mapsto v\ov$ and $v\mapsto \ov v$.  
\end{enumerate}
 The left transvection $u\mapsto \vp u$ for $u$ a leaf is the right transvection composed with a partial conjugation of type (1), so these are not included in our list.  We remark that the only nodes of $\G$ which produce non-trivial automorphisms of type (1) are those which separate $\G$.  Note also that if $\G$ has no squares, then there are no type (3) transvections.

\begin{lemma} 
The generators of $G(e_0,T_0)$ all commute. 
\end{lemma}

\begin{proof} Transvections affecting different nodes commute, since they are transvections by nodes in $\G_0$ onto nodes not in $\G_0$.
The two transvections affecting an interior node of $\G-\G_0$ also commute since they operate on different sides of the node.    The transvections affecting a single leaf are by commuting generators, so all of these transvections commute.

If $v_1$ and $v_2$ are separating nodes,  partial conjugations by $\vp_1$ and $\vp_2$ have disjoint support  if $\vp_1$ and $\vp_2$ are not comparable in the tree $T_0$, i.e. if the arc between them is not totally oriented.  If $\vp_1$ and $\vp_2$ are comparable, then the supports of the related partial conjugations are nested, so by composing with an inner automorphism of the larger one, their supports can be made disjoint, and they commute.

Since $v$ and $\vp$ commute, transvections by $v$ or $\vp$ onto a leaf $u$ at $v$ commute with partial conjugation by $v$ or $\vp$.  For partial conjugations by $\wplus \neq v,\vp$,
 either $u$, $v$ and $\vp$ are all in the support  of the partial conjugation or none of them are.  If they are all in the support, we can compose the partial conjugation by an inner automorphism to make the support disjoint. 

If $v$ is an interior node of $\G-\G_0$, then $v$ and $v_0$ belong to a single piece, so either both are affected by one of the  partial conjugations on our list or neither is.  
\end{proof}

The maximal 2-connected subgraphs of $\G$ decompose $\G$ into subgraphs, any two of which intersect at most in a single node.   These subgraphs are called the {\it pieces} of $\G$. (This terminology comes from the theory of tree-graded spaces \cite{DrSa}; the graph $\G$ is tree-graded with respect to its pieces.)

Let $v$ be a node of $\G$, and $\Delta(v)$ the union of all pieces of $\G$ containing $v$.   If every node of $\Delta(v)$ is either adjacent to $v$ or is $\leq v$, we say $v$ is a  {\it hub} (see Figure~\ref{hub}).   
If $\G$ is a tree then $\Delta(v) = st(v)$ for all interior nodes, so both endpoints of $e_0$ are hubs.

\begin{figure}\label{hub}
\labellist
\small\hair 2pt
\pinlabel{$v$} at 147 88
\pinlabel {$\Delta(v)$} [r] at 155 25
\endlabellist
\begin{center}
\includegraphics[width=2.5in]{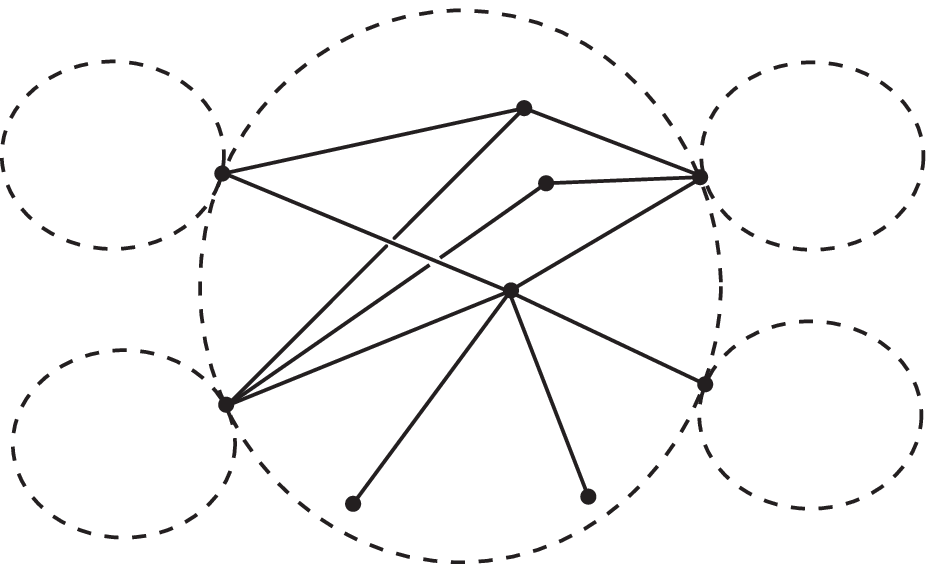}
\end{center}
\caption{A hub $v$} 
\end{figure}

\begin{theorem}\label{LBgeneral} Let  $\pi$ be the number of pieces of $\G$, $\nu$ the number of nodes in $\G$ and $\nu_0$  the number of nodes in $\G_0$.  Then
 $$\hbox{\sc vcd}(\out(\AG))\geq (\pi-1)+2(\nu - \nu_0)-2$$ 
If some node of $\G_0$ is not a hub, then
 $$\hbox{\sc vcd}(\out(\AG))\geq (\pi-1)+2(\nu - \nu_0)-1.$$ 
If some edge of $\G_0$ has no endpoints which are hubs, then 
 $$\hbox{\sc vcd}(\out(\AG))\geq (\pi-1)+2(\nu - \nu_0).$$ 
 \qed
\end{theorem}
 
 \begin{proof}
 The rank of $G=G(T_0,e_0)$ gives a lower bound for the VCD of $\out(\AG)$.  In general, the generators are linearly independent so $G$ has rank $(\pi-1)+ 2(\nu-\nu_0)$, where $\pi$ is the number of pieces of $\G$, $\nu$ is the number of nodes in $\G$ and $\nu_0$ the number of nodes in $\G_0$.  
However, in certain circumstances $G$ may contain as many as two independent inner automorphisms.   Note that every generator of $G$ acts trivially on the endpoints $v_0$ and $w_0$ of $e_0$.  Thus, the only inner automorphisms which can occur are conjugations by an element of $\langle v_0, w_0\rangle$.

Suppose that an  endpoint  $v_0$ of $e_0$ is a hub. Then the generators of $G$ include conjugation of any component of  $\G-\Delta(v_0)$ by $v_0$, as well as left and right transvections of $v_0$ onto any node of  $\Delta(v_0)$ which is not adjacent to $v_0$.  Thus the subgroup generated by these automorphisms includes conjugation by $v_0$.  If $v_0$ is not a hub, then some piece containing $v_0$ also contains a node which cannot be conjugated in $G$ by $v_0$.  

 \end{proof}
We record some special cases as corollaries:
 
 \begin{corollary}\label{LBtree} If $\G$ is a tree, then 
  $\hbox{\sc vcd}(\out(\AG))\geq e+2\ell-3,$ where $e$ is the number of edges and $\ell$ is the number of leaves.
\end{corollary}
\begin{proof}
If $\G$ is a tree, then the pieces are the edges of $\G$ and $\G_0$ is obtained by removing the leaves of $\G$, so the first estimate in Theorem~\ref{LBgeneral} gives the result.
\end{proof}

If $\G_0$ contains a cycle, then no node in that cycle can be a hub, so the third estimate in Proposition~\ref{LBgeneral} applies.  If $\G$ has no triangles or squares but is not a tree, then $\G_0$ is obtained from $\G$ by pruning off all leaves, and
we are in this situation.  More generally, we have:

\begin{corollary}\label{LBnosquares}  If $\G_0$ is obtained from $\G$ by pruning its leaves, but $\G$ is not a tree, then 
  $ \hbox{\sc vcd}(\out(\AG))\geq \pi+2\ell-1,$ where $\pi$ is the number of pieces and $\ell$ is the number of leaves.
\end{corollary}

\begin{corollary}\label{LBzero} If $\chi(\G)=0$ and the unique simple cycle $C$ of $\G$ has length $k\geq 5$, then 
$\hbox{\sc vcd}(\out(\AG))\geq  e-k+ 2\ell,$  where $e$ is the number of edges of $\G$ and $\ell$ is the number of leaves.
\end{corollary} 
\begin{proof} Since $\G$ has Euler characteristic $0$, it contains a unique simple cycle $C$ with $k$ sides, which constitutes one of its pieces.  The other pieces are single edges, contained in trees attached to the nodes of $C$, so the number of pieces is $e-k+1$.  Plugging this into Corollary~\ref{LBnosquares} gives the result.    
\end{proof}

\subsection{Upper bound} To get an upper bound on the VCD, we will use the projection homomorphisms  defined in \cite{CCV}.  As above, we assume $\G$ is connected and triangle-free.  Let  $\aut^0(\AG)$ be the subgroup of $\aut(\AG)$ generated by inversions, partial conjugations and transvections, and  let $\out^0(\AG)$ be its image in $\out(\AG)$. 
 Then  for  each node $v$ of $\G_0$, there is a homomorphism $P_v\colon \out^0(\AG)\to \out(F(L_v))$ where $F(L_v)$ is the free group on the nodes adjacent to $v$.  These are assembled to give a homomorphism 
\[
P\colon \out^0(\AG) \to \prod_{v \in V_0} \out(F(L_v)).
\]

The homomorphism $P$ is not onto; in fact the image of each $P_v$ lies in a subgroup of $\out(F(L_v))$ which we now describe.  
 
\begin{definition} If $N$ is a finite set and $K$ is a subset, we denote by $P\Sigma(N,K)$ the subgroup of $\out(F\langle N\rangle)$  generated by automorphisms which send each element of $K$ to a conjugate of itself. 
\end{definition}
  If $K$ is not empty,  we will prove in the following section that the VCD of $P\Sigma(N,K)$ is equal to $2n-k-2$, where $n$ is the cardinality of $N$ and $k$ is the cardinality of $K$ (Theorem~\ref{PSigma}).
If $K=\emptyset$, then $P\Sigma(N,\emptyset)=\out(F\langle N\rangle)$, which has VCD equal to $2n-3$ \cite{CuVo}.

Let $V_0$ be the nodes in $\G_0$ and let $U\subseteq V_0$ be the set of nodes of $\G$  which are maximal and unique in their equivalence class.  For each node $v$,   set $|v|=|L_v|=$ the valence of $v$, and   $|v|_U = |L_v\cap U|$.

 \begin{theorem}\label{UBgeneral} For any connected graph $\G$ with no triangles,
 $$\hbox{\sc vcd}(\out(\AG))\leq (\pi-1) +\sum_{v\in V_0}\big( 2|v|-3 - \max\{|v|_U-1, \, 0\}\big),$$
 where $\pi$ the number of pieces in $\G$.   
\end{theorem}
\begin{proof}
 By Proposition 3.11 of \cite{CCV}, the kernel $K_P$ of $P$ is free abelian of rank 
$$\sum_{v \in V_0}  (\delta_C(v) -1)$$
where $\delta_C(v)$ denotes the number of connected components of $\G-\{v\}$, or equivalently, the number of pieces containing $v$.  An easy induction argument on the number of pieces in $\G$ then shows that this sum equals $\pi-1$.

 If $v\in U$,  there are no transvections onto $v$, so any automorphism of $\aut^0(\AG)$ sends $v$ to a conjugate of itself. Thus   the image of $P$ lies in the product of the groups $P\Sigma(L_v,L_v\cap U)$, $v \in V_0$.  By Theorem \ref{PSigma}, these groups have VCD  equal to $2|v|-|v|_U-2$ if $L_v\cap U$ is not empty, and $2|v|-3$ if $L_v\cap U$ is empty.  Adding these to the VCD of the kernel of $P$ gives the upper bound of the theorem.
\end{proof}

If $U=V_0$, the statement simplifies as follows:

\begin{corollary}\label{UBnext} Let $\G$ be connected with no triangles, and assume $U=V_0$, i.e. every maximal vertex is unique in its equivalence class.
Then
 $$\hbox{\sc vcd}(\out(\AG))\leq (\pi-1) +2\big(\sum_{v\in V_0} |v| \big) - 2|V_0|-2|E_0|,$$
where $V_0$ and $E_0$ are the sets of nodes and edges  in $\G_0$  and $\pi$ is the number of pieces in $\G$.   

\end{corollary}
\begin{proof}Let $|v|_0$ denote the valence of $v$ in $\G_0$.  Because $U=V_0$ and  $\G_0$ is connected, $L_v\cap U$ is non-empty for all $v\in V_0$.  Thus, the latter terms  in the formula from Theorem~\ref{UBgeneral} become

\begin{align*}
\sum_{v\in V_0} \big(2|v|-3 - \max\{|v|_U-1, \, 0\}\big) &= \sum_{v\in V_0}\big( 2|v|-3 - |v|_0+1 \big)\\
&= 2\big(\sum_{v\in V_0} |v| \big) - 2|V_0|- 2|E_0|.
\end{align*}
\end{proof}

If $\G_0$ is obtained from $\G$ by pruning leaves, then $U=V_0$ since $V_0$ contains only one representative from each maximal equivalence class.   This holds, for instance, if $\G$ has no cycles of length less than $5$.  The statement in this case becomes even simpler. 

\begin{corollary}\label{UBthree} Let $\G$ be connected with no triangles. Assume $\G_0$ is obtained from $\G$ by pruning all leaves. Then
 $$\hbox{\sc vcd}(\out(\AG))\leq (\pi-1) + 2\big(\ell-\chi(\G)\big),$$
where $\ell$ is the number of leaves in $\G$ and $\pi$ is the number of pieces.    
\end{corollary}

\begin{proof} Since all vertices not in $\G_0$ are leaves, the latter terms in the formula in Corollary~\ref{UBnext} become
\begin{align*}
 2\sum_{v\in V_0} |v| - 2|V_0|-2|E_0| &=  2\sum_{v\in V_0}( |v| - 1)-2|E_0|\\
&= 2\sum_{v\in V}( |v| - 1) - 2|E_0|\\
&=2\big(|E|-\chi(\G)- |E_0|\big)\\
&=2(\ell-\chi(\G)).
\end{align*}
\end{proof}

If $\G$ is a tree, then $\pi=e$ and $\G$  has Euler characteristic $1$, so our upper bound agrees with the lower bound given in Corollary~\ref{LBtree}: 

\begin{theorem}\label{Tree} If $\G$ is a tree, then
$\hbox{\sc vcd}(\out(\AG))= e+2\ell-3,$  where $e$ is the number of edges of the tree and $\ell$ is the number of leaves.
\end{theorem}

If $\G$ is not a tree, but has no cycles of length less than 5, the lower bound from Corollary~\ref{LBnosquares} and upper bound from Corollary~\ref{UBthree} combine to give

\begin{theorem}\label{NoShortCycles}   If $\G$ is connected with no triangles or squares but is not a tree, the VCD of $\out(\AG)$ satisfies  $$ \pi  + 2\ell-1 \leq\hbox{\sc vcd}(\out(A_\G)) \leq \pi+2\ell-1-2\chi(\G),$$  where $\ell$ is the number of leaves of $\G$, $\pi$ is the number of maximal 2-connected subgraphs, and $\chi(\G)$ is the Euler characteristic of $\G$.  
\end{theorem}

 If $\G$ has Euler characteristic $0$ and no squares, this gives:
 
\begin{theorem}\label{cycle} If $\chi(\G)=0$ and the unique simple cycle $C$ of $\G$ has length $k\geq 5$, then 
$$\hbox{\sc vcd}(\out(\AG))= e-k+2\ell.$$ 
\end{theorem} 

Theorems \ref{NoShortCycles} and \ref{cycle} also hold if the no squares condition is replaced by the requirement that $\G_0$ be obtained by pruning leaves.  For example, Theorem \ref{cycle}  holds for a square with a tree attached to each of its vertices.


\subsection{Further finiteness properties for tree-based RAAGs}

 In the case that $\G$ is a tree, we can identify the image of $P$ precisely:

  \begin{theorem}  Assume $\G$ is a tree and hence $V_0$ is the set of non-leaf nodes of $\G$.  
Then the image of $P$ is the product
\[
 \prod_{v \in V_0} P\Sigma(L_v, L_v\cap V_0)
 \]
 \end{theorem}
 
 \begin{proof}  It follows from \cite{CCV}, Propostion 3.2, that an element of $\out^0(\AG)$ must preserve every element of $V_0$ up to conjugacy, thus the image of $P$ lies in  $\prod P\Sigma(L_v, L_v\cap V_0)$.
 
 Let $\phi$ be an element of $P\Sigma(L_v, L_v\cap V_0)$.  To prove surjectivity, we will show that $\phi$ lifts to an element $\hat \phi$ in $\out^0(\AG)$ whose image under $P$ is $\phi$ in the $v$-factor and the identity in every other factor.  
 
 Fix a representative automorphism for $\phi$ (which by abuse of notation we also denote $\phi$).  For each $w \in L_v$, $\phi$ acts on $w$ as conjugation by some $g_w \in F(L_v)$.  Define $\hat \phi \in \out^0(\AG)$ as follows.  Set $\hat\phi (v)=v$ and for $w \in L_v$, set $\hat\phi (w)= \phi(w) = g_w w g_w^{-1}$.  For $u \notin L_v$, let $w \in L_v$ be the (unique) vertex of $L_v$ at minimum distance from $u$.   (Equivalently, $w$ is the unique element of $L_v$ such that $w$ and $u$ lie in the same connected component of $\G-\{v\}$.)  Define $\hat\phi (u)=g_w u g_w^{-1}$.  
 
 It is not difficult to see that $\hat\phi$ is an automorphism.  For suppose $\rho$ is the inverse of $\phi$ in $\aut(L_v, L_v \cap V_0)$ and say $\rho(w)=h_wwh_w^{-1}$, $h_w \in F(L_v)$,  for $w \in L_v$.  Then 
 \[
 \rho\circ\phi(w) = \rho(g_w wg_w^{-1})=\rho(g_w)h_w w h_w^{-1}\rho(g_w^{-1})=w,
 \]
hence we must have $\rho(g_w)h_w= w^s$  for some $s$.  It follows that for any $u$ in the same connected component of  $\G-\{v\}$, $\hat\rho\circ \hat\phi(u)= w^suw^{-s}$.  This composite is clearly an automorphism.
 
Finally, consider $P(\hat\phi)$.  By definition, it agrees with $\phi$ on the $v$-factor.  For any other node $u \in V_0$,  there is a component $C$ of $\G -\{v\}$  such that $L_u$ lies entirely in $C \cup \{v\}$.  If $w$ is the unique element of $C \cap L_v$, then $\hat\phi$ acts on every element of $L_u$ as conjugation by $g_w$.  (Here we use the fact that $g_w$ lies in $F(L_v)$ hence commutes with $v$.)  Hence its image in $P\Sigma(L_u, L_u\cap V_0)$ is the identity.
 \end{proof}

\begin{corollary} If $\,\G$ is a tree, then $\out^0(\AG)$ is virtually torsion free and every finite index subgroup has a finite Eilenberg-MacLane complex.  In particular, $\out^o(\AG)$ is finitely-presented, and has finitely-generated homology in all dimensions.  
\end{corollary}

\begin{proof} Recall from \cite{CCV} that the kernel $K_P$ of $P$ is a finitely generated abelian group.  The corollary follows from the same statements for $P\Sigma(n,k)$ (proved in \cite{CuVo})  and $K_P$, since the class of groups with finite Eilenberg-MacLane complexes is closed under extensions.   
\end{proof}


\section{Partially symmetric automorphisms of free groups}
Let $F_n$ be a free group with a fixed set of generators $x_1,\ldots,x_n$, and let $P\Sigma(n,k)$ be the subgroup of the outer automorphism group of  $F_n$ consisting of outer automorphisms which send the first $k$ generators to conjugates of themselves.  For $k=n$ this is known as the pure symmetric automorphism group, and Collins \cite{Col} showed that it has virtual cohomological dimension equal to $n-2$. The full outer automorphism group $\out(F_n)$ has virtual cohomological dimension $2n-3$.  We show that the virtual cohomological dimension of $P\Sigma(n,k)$ varies linearly with $s$ between these two values.  

\begin{theorem}\label{PSigma} For $k\geq 1$, $P\Sigma (n,k)$ has VCD equal to $2n-k-2$.
\end{theorem}

A lower bound for this VCD is given by the rank of the free abelian subgroup of $P\Sigma (n,k)$ generated by automorphisms $\gamma_i$ (for $1<i\leq k$) and $\lambda_i, \rho_i$ (for $k<i\leq n$), where
$$
{\gamma_i\colon\begin{cases}  x_i\mapsto x_1^{-1}x_i x_1&\cr
            x_j\mapsto x_j & j\neq i
            \end{cases}}
$$
and
$$
{\lambda_i\colon\begin{cases}  x_i\mapsto x_1x_i&\cr
            x_j\mapsto x_j & j\neq i
            \end{cases}}
\qquad\qquad{\rho_i\colon\begin{cases}  x_i\mapsto x_ix_1\cr
            x_j\mapsto x_j & j\neq i
            \end{cases}}
$$
            
These automorphisms generate a free abelian subgroup of $\aut(F_n)$ of rank $2n-k-1$.   This subgroup contains conjugation by $x_1$, so projects to a free abelian subgroup of $P\Sigma(n,k)$ of rank $2n-k-2$.

To determine the upper bound,  we will find a contractible $(2n-k-2)$-dimensional simplicial complex on which $P\Sigma(n,k)$ acts with finite stabilizers.  We  first recall some background about the spine of Outer space and its minimal subcomplexes.

\subsection{Background on the spine of outer space and minimal subcomplexes}

Recall that the {\it spine of Outer space} is a simplicial complex $K$ of dimension $2n-3$ on which $\out(F_n)$ acts with finite stabilizers and compact quotient \cite{CuVo}.  Vertices of $K$ are   {\it marked graphs} $(g,\G)$, where $\G$ is a graph with all nodes of valence at least 3 and and $g$ is a homotopy equivalence from a fixed rose $R_n$ to $\G$.    Marked graphs $(g,\G)$ and $(g',\G')$ are equivalent if there is a graph isomorphism $h\colon \G\to\G'$ with $h\circ g$ homotopic to $g'$.  Note that we are allowing separating edges in the graphs $\G$, i.e. we are considering the full Outer space, as opposed to reduced Outer space. 

Vertices $(g,\G)$ and $(g',\G')$ span an edge if $(g',\G')$ can be obtained from $(g,\G)$ by collapsing each tree in a forest $F\subset \G$ to a point, and $K$ is the associated flag complex.  In other language, forest collapse is a poset relation on the vertices of $K$, and $K$ is the geometric realization of this poset.  Height in the poset is given by the number of nodes of $\G$.
The minimal elements of the poset are  marked roses,  and $K$ is the union of the stars of these roses. 
The maximal elements are marked graphs with all nodes of valence 3.  Such a graph has 2n-2 nodes, hence the dimension of $K$ is $2n-3$.

(Ideal edges). The upper link of a vertex $(g,\G)$ in the poset consists of marked graphs $(g',\G')$ which collapse to $(g,\G)$.  Such marked graphs are said to be obtained by {\it blowing up} nodes of $\G'$ into trees.  We recall from \cite{CuVo} that a convenient way of describing blow-ups at a node $v$ is in terms of partitions of the set $H(v)$ of half-edges terminating at $v$.  An edge of the blow-up tree partitions the tree, and therefore $H(v)$, into two sets, each with at least two elements.  For this reason, such a partition is called an {\it ideal edge} at $v$.  A set of ideal edges is {\it compatible} if it corresponds to a tree.  Since blowing up can be done independently at different nodes of $\G$, the upper link of $(g,\G)$ is the join of subcomplexes $B'(v)$.  Each $B'(v)$ is the barycentric subdivision of the complex $B(v)$ whose vertices are ideal edges at $v$   and whose $i$-simplices are sets of $i+1$ compatible ideal edges.

(Norm and minimal subcomplexes).  Given a set $W$ of conjugacy classes in $F_n$, we can define a norm on roses as follows.  Represent each $w\in W$ by a cyclically reduced edge-path in $R_n$, and define $|w|$ to be the length  of the (cyclically reduced) edge-path $g(w)$ in $\Gamma$.  Then  
$$\|\rho\|=\sum_{w\in W} |g(w)|.$$
Stars of roses of minimal norm with respect to $W$ form a subcomplex of $K$ which we denote $\KW$.  The main theorem of \cite{CuVo} is:
\begin{CVtheorem} [Culler-Vogtmann]  Let $W$ be a set of cyclic words in $F_n$.  Then the associated subcomplex $\KW$ of $K$  is contractible, the action is proper and the quotient $\KW/Stab(W)$ is finite.  
\end{CVtheorem}
 Recall that $\out(F_n)$ acts on $K$ as follows.  Realize $\phi\in \out(F_n)$ by a homotopy equivalence $f\colon R_n\to R_n$; then $(g,\Gamma)\cdot\phi = (g\circ f,\Gamma)$. In particular, the stabilizer of $W$ in $\out(F_n)$ preserves the norm, so that  $\KW$ is invariant under the action of this stabilizer. 

\subsection{A complex for $P\Sigma(n,k)$}

We now  fix $W$ to be the conjugacy classes of $\{x_1,\ldots,x_k\}$.  Let $\KS=K_W$ denote the minimal subcomplex for $W$ and $\out(n,k)$ the stabilizer of $W$ in $\out(F_n)$.  Since $P\Sigma (n,k)$ has finite index in $\out(n,k)$  it has the same VCD, and we will use the action of $\out(n,k)$ on $\KS$ to compute this. 

We first record some simple observations about the vertices of $\KS$. For a vertex $(g,\G)$ in $\KS$,  define a forest $F\subset \G$ to be {\it admissible} if the marked graph $(g\circ c_F, \G_F)$ obtained by collapsing each tree in $F$ to a point is also in $\KS$.

\begin{lemma}\label{forest} Let $(g,\Gamma)$ be a vertex of   $\KS$, and let $C_i$ be the reduced edge path representing $g(x_i)$, for $1\leq i\leq k$.  Then
\begin{enumerate}
\item $C_i$ is  a simple cycle in $\Gamma$, for each $i=1,\ldots,s$.
\item $C_i\cap C_j$ is either empty,  a point or  a connected arc.
\item  $\Phi=\cup(C_i\cap C_j)$ is a forest in $\Gamma$.  
\item  If $F$ is an admissible forest in $\Gamma$, then $F\cup \Phi$ is an admissible forest.  
\end{enumerate}
\end{lemma}

\begin{proof}  Let $\rho$ be any marked rose in $\KS$ with $(g,\G)$ in its star. These statements all follow from the observation that, in $\rho$, the image of $x_i$ is a single petal.  Since $(g,\G)$ is obtained from $\rho$ by blowing up the node of $\rho$ into a tree $T$, the intersection of any two $C_i's$ is contained in $T$, so the union of all such intersections is a forest in $T$ (see Figure 2).
Since this is true for any  choice of $\rho$,  this union is in all maximal admissible forests. 
\end{proof}

\labellist
\small\hair 2pt
\pinlabel{$C_1$} at 60 40
\pinlabel{$C_2$} at 30 110
\pinlabel{$C_3$} at 100 140
\pinlabel{$C_4$} at 130 110
\pinlabel{$C_1$} at 300 35
\pinlabel{$C_2$} at 360 130
\pinlabel{$C_3$} at 420 100
\pinlabel{$C_4$} at 430 42
\endlabellist
\begin{center}
\includegraphics[width=4in]{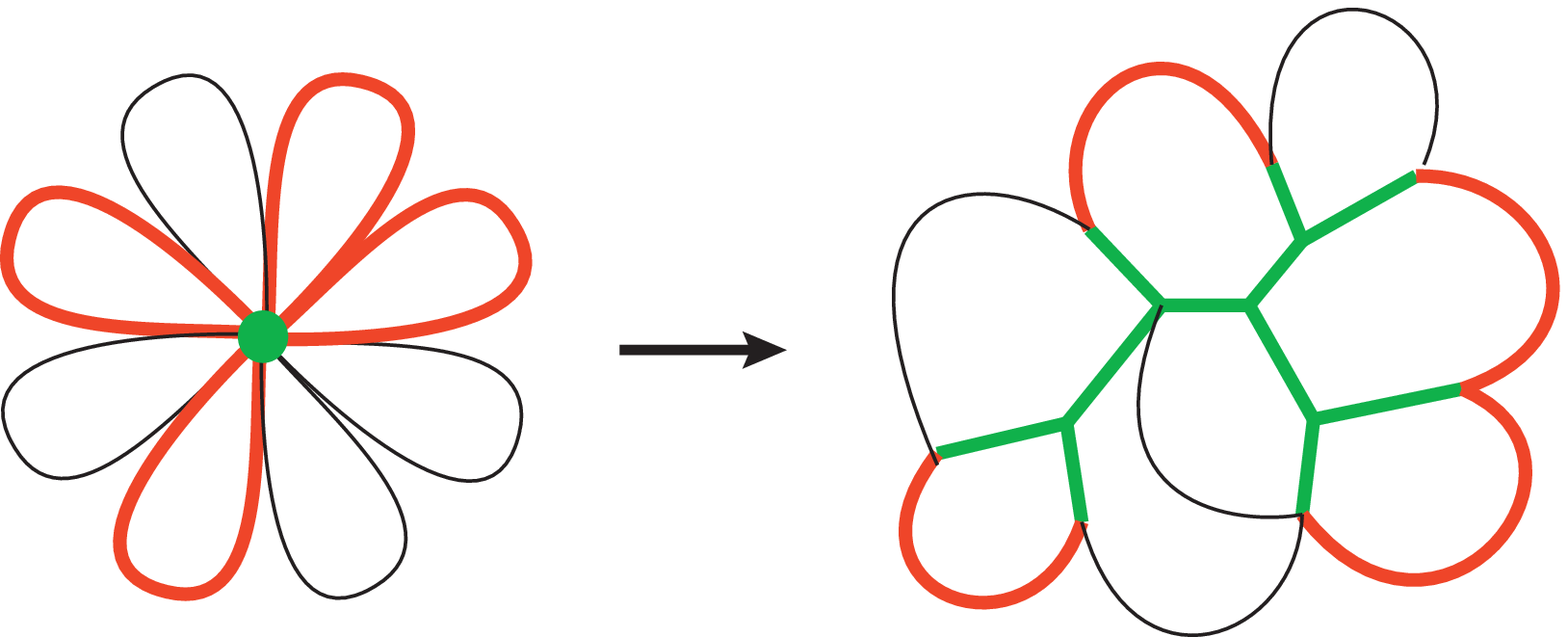}

{Figure 2.  The cycles $C_i$ in a blowup  of $\rho$ } 
\end{center}
\medskip

The dimension of  $\KS$ is equal to $2n-3$ for any $k$.  This is the dimension we need for $k=1$, but is too big for $k>1$.   Instead we consider the following subcomplexes.

\begin{definition} Let $\KO$ be the subcomplex $\KS$ spanned by vertices $(g,\G)$ in which any two of the cycles $C_1,\ldots,C_k$  intersect at most in a point. 
\end{definition}

\begin{definition} Let $\LO$ be the subcomplex $\KS$ spanned by vertices $(g,\G)$ in which all of the cycles $C_1,\ldots,C_k$ are disjoint.  
\end{definition}

Any marked graph in $\LO$ has at least $k$ vertices (one on each $C_i$), and at most $2n-2$, so $dim(\LO) = 2n-k-2$ (for $k=n$, $dim(\KO)$ is also equal to $n-2$).  This dimension is equal to our lower bound on the VCD of $\out(n,k)$, so it remains only to prove the following two propositions.

\labellist
\small\hair 2pt
\pinlabel{$C_1$} at 30 55
\pinlabel{$C_2$} at 90 150
\pinlabel{$C_3$} at 150 120
\pinlabel{$C_4$} at 160 62
\pinlabel{$C_1$} at 370 35
\pinlabel{$C_2$} at 430 145
\pinlabel{$C_3$} at 490 130
\pinlabel{$C_4$} at 500 80
\endlabellist
\begin{center}
\includegraphics[width=4in]{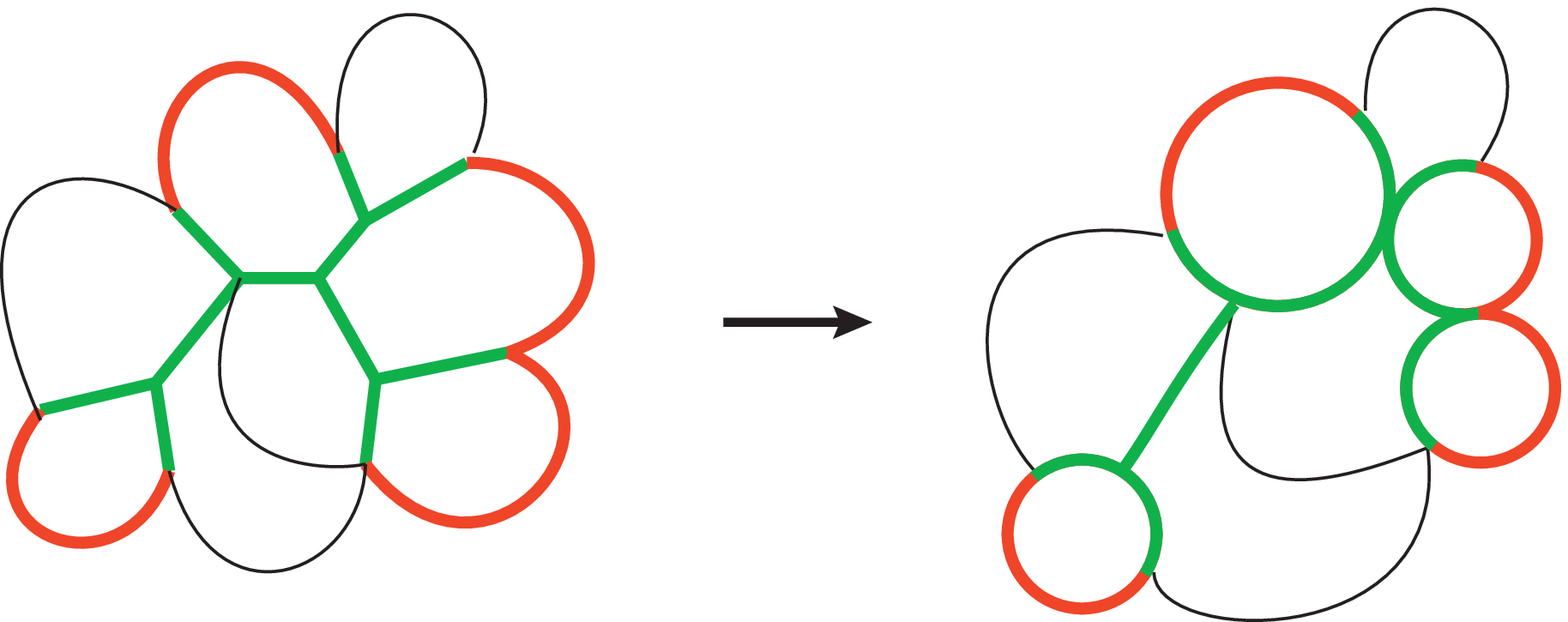}
\end{center}
\centerline{Figure 3. Deforming $\KS$ to $\KO$ } 
\medskip

\begin{proposition}\label{points} There is an $\out(n,k)$-equivariant deformation retraction of $\KS$ onto  $\KO$. 
\end{proposition} 

\begin{proof} Let $(g,\G)$ be any vertex in $\KS$.  We  perform the deformation retraction of $\KS$ to $\KO$ by collapsing each component of $\Phi=\cup(C_i\cap C_j)$ to a point (See Figure 3).  If $(g',\Gamma')$ in $\KS$ is obtained from $(g,\Gamma)$ by collapsing a forest $F$, then $F\cup \Phi$ is also a forest in $\Gamma$ by part (3) of Lemma~\ref{forest}; thus collapsing $\Phi$ is a poset map, so gives a deformation retraction onto its image.
\end{proof}

\labellist
\small\hair 2pt
\pinlabel{$C_1$} at 35 32
\pinlabel{$C_2$} at 100 140
\pinlabel{$C_3$} at 165 130
\pinlabel{$C_4$} at 170 75
\pinlabel{$C_1$} at 330 40
\pinlabel{$C_2$} at 395 155
\pinlabel{$C_3$} at 510 135
\pinlabel{$C_4$} at 480 42
\endlabellist
\begin{center}
\includegraphics[width=4in]{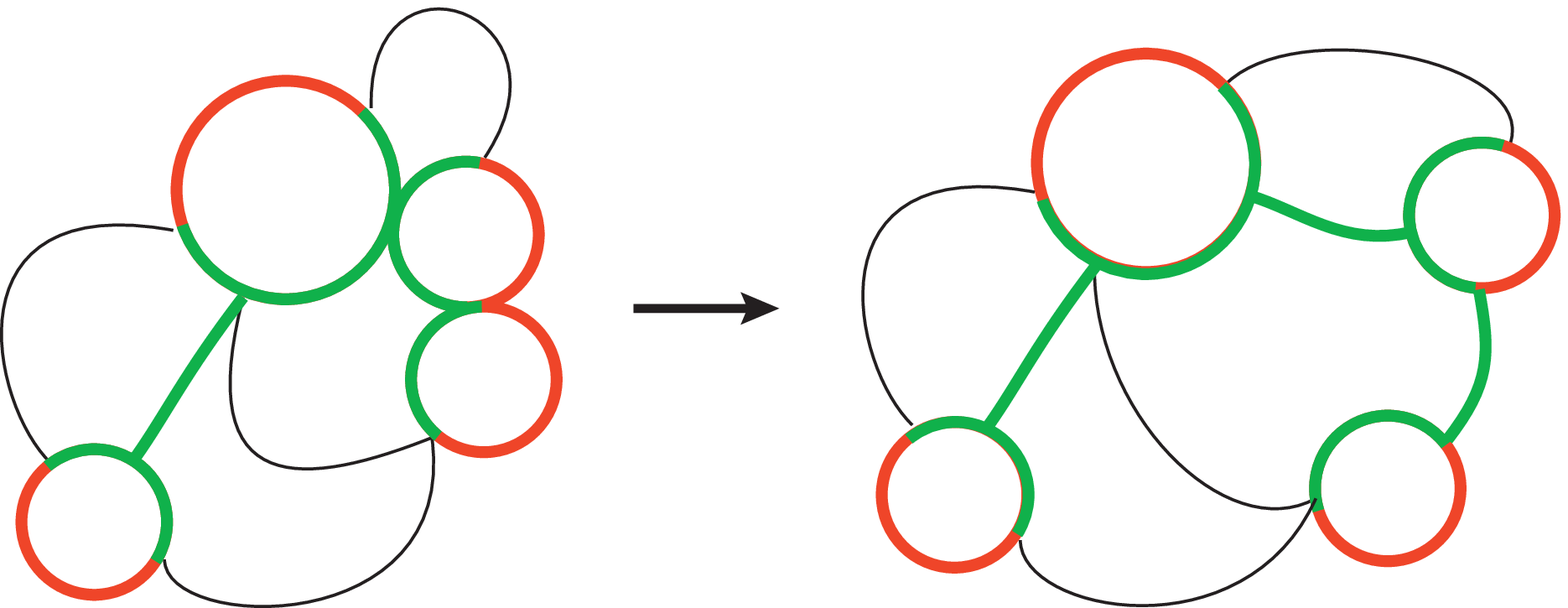}
\end{center}
\centerline{Figure 4. Deforming $\KO$ to $\LO$ }
\medskip

\begin{proposition}\label{disjoint} There is an $\out(n,k)$-equivariant deformation retraction of $\KO$ onto  $\LO$. 
\end{proposition}

\begin{proof}
  
We build $\KO$ from $\LO$ by adding vertices $(g,\G)$ in order of {\it decreasing} height in the poset, i.e. decreasing number of nodes $n(\G)$. Thus at each stage, we are attaching $(g,\G)$ along its entire upper link in $\KO$, so it suffices to show this upper link is contractible.  Note that  all trivalent marked graphs in $\KO$ already belong to $\LO$.  

We use the description of the upper link as the join of complexes $B(v)$ of ideal edges at nodes $v$ of $\G$, but we are only interested in blow-ups which result in vertices of $\KO$, so we 
define an ideal edge to be {\it legal} if blowing it up results in a vertex of $\KO$.  Note that an ideal edge is legal if and only if it separates at most one pair of half-edges contained in a cycle $C_i$, since if it separates two pairs $x_i,\bar x_i$ and $x_j,\bar x_j$, it blows up to an edge in $C_i\cap C_j$.  
Define $L(v)$ to be the subcomplex of $B(v)$ spanned by legal ideal edges.

 Let $(g,\G)$ be a vertex of $\KO-\LO$.  Then $\G$ contains a node $v$ which is in at least two cycles, so to prove that the upper link of $(g,\G)$ in $\KO$ is contractible, it suffices to prove the following.

\begin{claim} If $v$  is contained in at least 2 cycles, then the complex  $L(v)$ of legal ideal edges at $v$ is contractible.
\end{claim}
\begin{proof} 
The set of half-edges $H(v)$ at $v$ is the union of half-edges $A=\{a_1,\bar a_1,\ldots,a_r,\bar a_r\}$ contained in some $C_i$ and $B=\{b_1,\ldots,b_s\}$ not contained in any $C_i$. Fix an element $a\in A$ and define the {\it inside} of an ideal edge to be the side containing $a$, and the {\it size} to be the number of half-edges on the inside.  

 We are assuming $r\geq 2$.   We prove the claim by induction on $s$, starting with $s=0$.

If $s=0$,  consider the ideal edge $\alpha_0$ which separates $a$ and $\bar a$ from all other half-edges.  Let $C(\alpha_0)$ denote the star of $\alpha_0$ in $L(v)$.
We show that $L(v)$ deformation retracts onto $C(\alpha_0)$, and is therefore contractible, by adding vertices of $L(v)-C(\alpha_0)$ in order of increasing size.  We may think of size as a Morse function on $L(v)-C(\alpha_0)$, and extend it to $C(\alpha_0)$ by setting it equal to $0$ on vertices of $C(\alpha_0)$.    We need to show that the descending link with respect to this Morse function is contractible, for each vertex of $L(v)-C(\alpha_0)$. 

A vertex $\alpha$ of $L(v)$ which is not  in $C(\alpha_0)$ is a legal ideal edge which separates $a$ from $\bar a$.  Note that $\alpha$ must have odd size at most $2r-3$. Let $I$ be the inside of $\alpha$, and let $\alpha^+$ be the ideal edge with inside $I\cup \bar a$.   Then $\alpha^+$ is still legal, and  the descending link of $\alpha$ is a  (contractible) cone with cone point  $\alpha^+$, since any ideal edge of smaller size which is compatible with $\alpha$ is also compatible with $\alpha^+$.

Now consider the case $s\geq 1$, and write $L(r,s)$ for $L(v)$.   Choose  $b\in B$, and let $\beta_0$ be the ideal edge separating $a$ and $b$ from all other half-edges.  As before, let $C(\beta_0)$ be the star of $\beta_0$ in $L(r,s)$ and use size to define a Morse function.  If $J$ is the inside of $\beta$, let $\beta^+$ be the ideal edge with inside $J\cup b$. As before, the descending link is a cone with cone point $\beta^+$. The only difference comes at the very end, since maximal $\beta$ have size $2r+s-2$, in which case $\beta^+$ has a singleton on one side, so is not an ideal edge.

This is where induction is used.  Let $\beta$ be a vertex in $L(r,s)-C(\beta_0)$ of size $2r+s-2$.  The outside of $\beta$ is either $\{b,b'\}$ ($b'\neq b$), $\{b,\bar a\}$ or $\{b,a'\}$ ($a'\neq a,\bar a$), and its descending link is its entire link.  In each case, this link can be identified with $L(r,s-1)$, so is contractible by induction.  
\end{proof}
This completes the proof of the proposition.
 \end{proof}

\def\cprime{$\prime$}

\end{document}